\documentclass[11pt,reqno]{amsart}

\usepackage{amsmath,amsthm,amssymb}

\newtheorem{theorem}{Theorem}[section]
\newtheorem{lemma}[theorem]{Lemma}
\newtheorem{proposition}[theorem]{Proposition}
\newtheorem{corollary}[theorem]{Corollary}

\theoremstyle{definition}

\newtheorem{remark}[theorem]{Remark}

\numberwithin{equation}{section}

\begin{document}

\title[On the non-existence of left-invariant hypercomplex $\dots$]{On the non-existence of left-invariant hypercomplex structures on $SU(2)^{4n}$}

\author{David N. Pham}
\address{Department of Mathematics $\&$ Computer Science, QCC CUNY, Bayside, NY 11364}
\curraddr{}
\email{dnpham@qcc.cuny.edu}



\dedicatory{}

\begin{abstract}
Using elementary algebraic arguments, it is shown that $SU(2)^{m}:=SU(2)\times \cdots \times SU(2)$ ($m$ times)  admits no left-invariant hypercomplex structures for all $m\ge 1$.  This result answers (in a clear and easily accessible way) the question of whether every compact Lie group of dimension $4n$ admits a left-invariant hypercomplex structure.  The aforementioned question has apparently been the source of some confusion in the recent literature.
\end{abstract}


\maketitle

\section{Introduction}
A hypercomplex structure on a manifold $M$ is a triple $(I,J,K)$ of integrable almost complex structures on $M$ which all anti-commute with one another and satisfy $IJ=K$.  A necessary condition\footnote{See Proposition \ref{propHypercomplexDimension} and Corollary \ref{corHypercomplexDimension}.} for the existence of a hypercomplex structure is that the dimension of the manifold must be a multiple of $4$. This is entirely analogous to the statement that all complex manifolds are necessarily of even dimension.  Of course, dimensional considerations alone are not sufficient to guarantee the existence of a complex or hypercomplex structure.

Amazingly, it was shown by Samelson \cite{Samelson53} and, independently, by Wang \cite{Wang54} that every compact Lie group of even dimension admits a left-invariant complex structure.  It is quite tempting to imagine that an analogous statement holds for compact Lie groups of dimension $4n$ as it pertains to hypercomplex geometry.  In other words, it is natural to conjecture that every compact Lie group of dimension $4n$ admits a left-invariant hypercomplex structure.  Spindel, Servin, and Troost in \cite{Spindel88} (from the physics perpsective) and Joyce in \cite{Joyce92} (from the mathematical perspective) both showed that if $G$ is any compact Lie group, then $G\times T$ admits a left-invariant hypercomplex structure, where $T$ is a sufficiently large torus.  This construction lends some plausibility to the aforementioned conjecture.

As it turns out, the powerful machinery presented in the work of Dimitrov and Tsanov in \cite{Dimitrov2016} can be brought to bear to resolve the above conjecture in the negative. However, the machinery of \cite{Dimitrov2016} relies on rather sophisticated Lie theory and combinatorics which makes it somewhat difficult to digest, especially for non-experts. This, we believe, is part of the reason why there appears to be some confusion in the literature with regard to the question of whether every compact Lie group of dimension $4n$ admits a left-invariant hypercomplex structure (even among experts in complex geometry).  For example, as recently as \cite{Andrada2023}, the following (incorrect) assertion is made in the paper's Introduction: ``\textit{Indeed, Joyce showed in [40] that any compact Lie group of dimension $4n$ admits a left-invariant hypercomplex structure.}".    

In the current paper, we resolve the confusion regarding the aforementioned question by using elementary algebraic arguments to establish the following result:
\begin{theorem}
\label{thmMainTheorem}
$SU(2)^m$ admits no left-invariant hypercomplex structures for all $m\ge 1$.
\end{theorem}
\noindent where $SU(2)^m:=SU(2)\times \cdots \times SU(2)$ ($m$ times).  Of course, since hypercomplex geometry requires the dimension of the manifold to be a multiple of $4$, the cases of interest in Theorem \ref{thmMainTheorem} occur for $m=4n$.  It is our hope that by relying only on straightforward algebraic arguments to establish Theorem  \ref{thmMainTheorem}, we definitively eliminate any confusion surrounding the aforementioned question while making the result accessible to a wide audience.

\section{Preliminaries}
\label{secPrelim}
In this section, we establish the notation for the current paper while reviewing some basic facts about hypercomplex manifolds for the beneift of the reader who has no familiartiy with hypercomplex geometry.
\begin{proposition}
\label{propHypercomplexDimension}
If $(I,J,K)$ is a hypercomplex structure on a real, finite dimensional vector space $V$, then $\dim V=4n$ for some integer $n\ge 0$.
\end{proposition}
\begin{proof}
    Assume that $\dim V>0$ and let $\eta$ be any (positive definite) inner product on $V$.  Define $\langle\cdot,\cdot\rangle$ to be the inner product on $V$ given by
    $$
    \langle u,v\rangle:=\eta(u,v)+\eta(Iu,Iv)+\eta(Ju,Jv)+\eta(Ku,Kv).
    $$
    From this, is easy to see that $\langle\cdot,\cdot\rangle$ is invariant under $I$, $J$, and $K$. Let $v$ be any nonzero element of $V$. It is easy to see that $\{v,Iv,Jv,Kv\}$ are orthogonal with respect to $\langle\cdot, \cdot\rangle$.  In particular, the aforementioned vectors are linearly independent.  Let $U_1\subset V$ be the $4$-dimensonal subspace with basis $\{v,Iv,Jv,Kv\}$ and let $V_1:=U_1^\perp$ be the orthogonal complement of $U_1$ in $V$ with respect to $\langle\cdot,\cdot\rangle$.  Then $V=U_1\oplus V_1$. The $(I,J,K)$-invariance of $\langle\cdot,\cdot\rangle$ together with the (clear) invariance of $U_1$ under $I$, $J$, and $K$ implies that $V_1$ is also invariant under $I$, $J$, and $K$.  Now, if $\dim V_1>0$, we can repeat the entire process with $V_1$ and extract another $4$-dimensional subspace $U_2\subset V_1$ to obtain the decomposition
    $$
    V=U_1\oplus U_2\oplus V_2,
    $$
    where $V_2:=U_2^\perp$ is the orthogonal complement of $U_2$ in $V_1$ with respect to $\langle\cdot,\cdot\rangle$.  If $\dim V_2>0$, we repeat the process again and so on.  Ultimately, we arrive at a decomposition of the form
    $$
    V=U_1\oplus U_2\oplus \cdots \oplus U_n,
    $$
    where $\dim U_i=4$ for $i=1,\dots, n$ for some $n$.  Hence, $\dim V=4n$.   
\end{proof}
\begin{corollary}
\label{corHypercomplexDimension}
If $(M,I,J,K)$ is a hypercomplex manifold, then $\dim M=4n$ for some integer $n\ge 0$.
\end{corollary}
Now let $G$ be a Lie group and let $\mathfrak{g}:=\mbox{Lie}(G)$ be its Lie algebra of left-invariant vector fields (which we identify with $T_eG$, the tangent space of $G$ at the identity element $e$).  Recall that a complex structure $J$ on $G$ is left-invariant if the left translation map $L_g: G\rightarrow G$, $x\mapsto gx$ is holomorphic with respect to $J$ for all $g\in G$.  In other words, $J$ commutes with the pushforward (or differential) of $L_g$:
$$
(L_g)_\ast \circ J=J\circ (L_g)_\ast.
$$
The above condition implies that $J$ is entirely determined by its value at $e$, and, moreover, $J$ maps left-invariant vector fields to left-invariant vector fields. Hence, $J$ can be identified with its restriction to $\mathfrak{g}:=T_eG$: $$
J: \mathfrak{g}\rightarrow \mathfrak{g},
$$
where we use the same symbol $J$ for the restriction to $\mathfrak{g}$.

By the celebrated Newlander-Nirenberg theorem, $J$ is integrable if and only if the Nijenhuis tensor 
$$
N_J(X,Y):=J[JX,Y]+J[X,JY]+[X,Y]-[JX,JY]
$$
vanishes for all $X,Y\in \mathfrak{X}(G)$.  However, since $J$ is left-invariant, we see that $N_J=0$ if and only if $N_J(X,Y)=0$ for all $X,Y\in \mathfrak{g}$. In light of the above remarks, we will simply define a left-invariant integrable almost complex structure on $G$ to be a linear map $J:\mathfrak{g}\rightarrow \mathfrak{g}$ satisfying the following conditions:
$$
J^2=-\mbox{id}_{\mathfrak{g}},~\hspace*{0.1in}N_J(X,Y)=0,~\forall~X,Y\in \mathfrak{g}.
$$

Let $\mathfrak{su}(2)$ denote the Lie algebra of $SU(2)$.  A convenient basis for $\mathfrak{su}(2)$ is
$$
e_1:=\frac{1}{2}\left(\begin{array}{cc}
i & 0\\
0 & -i
\end{array}\right),\hspace*{0.1in}e_2:=\frac{1}{2}\left(\begin{array}{cc}
0 & i\\
i & 0
\end{array}\right),\hspace*{0.1in}e_3:=\frac{1}{2}\left(\begin{array}{cc}
0 & -1\\
1 & 0
\end{array}\right).
$$
The nonzero bracket relations are then 
$$
[e_1,e_2]=e_3,\hspace*{0.1in}[e_2,e_3]=e_1,\hspace*{0.1in} [e_3,e_1]=e_2.
$$
For completeness, we record the following basic fact:
\begin{proposition}
    \label{propSU2XY}
    Let $X,Y\in \mathfrak{su}(2)$.  Then $[X,Y]\neq 0$ if and only if $X$ and $Y$ are linearly independent.  In addition, if $[X,Y]\neq 0$, then $\{X,Y,[X,Y]\}$ is a basis of $\mathfrak{su}(2)$.
\end{proposition}
\begin{proof}
    Write 
    $$
        X=\sum_i a_i e_i,\hspace*{0.1in} Y=\sum_i b_ie_i.
    $$
    By direct calculation, we have
    $$
    [X,Y]=\mbox{det}\left(\begin{array}{cc}
    a_2 & a_3\\
    b_2 &  b_3
    \end{array}\right)e_1+\mbox{det}\left(\begin{array}{cc}
    a_3 & a_1\\
    b_3 &  b_1
    \end{array}\right)e_2+\mbox{det}\left(\begin{array}{cc}
    a_1 & a_2\\
    b_1 &  b_2
    \end{array}\right)e_3.
    $$
    $X$ and $Y$ are then linearly independent precisely when at least one of the coefficients in the above expression is nonzero.  In other words, $X$ and $Y$ are linearly independent precisely when $[X,Y]\neq 0$.

    For the last statement, one recognizes that the coefficients of the components of $[X,Y]$ are simply the components of the (right-hand) cross product of the vectors $(a_1,a_2,a_3)$ and $(b_1,b_2,b_3)$ in $\mathbb{R}^3$.  Hence, if $[X,Y]\neq 0$, then $\{X,Y,[X,Y]\}$ must be linearly independent and therefore a basis of $\mathfrak{su}(2)$.
\end{proof}
As observed in the proof of Proposition \ref{propSU2XY}, $\mathfrak{su}(2)$ is isomorphic to the Lie algebra whose underlying vector space is $\mathbb{R}^3$ with Lie bracket given by the (right-hand) cross product.  Using this interpretation, one immediately obtains the following:
\begin{corollary}
    \label{corBracketsLI}
    If $\{X,Y,Z\}$ is a basis of $\mathfrak{su}(2)$, then $\{[X,Y],~[Y,Z],~[Z,X]\}$ is also a basis of $\mathfrak{su}(2)$.
\end{corollary}

We conclude this section by introducing some additional notation that will be convenient for what follows.  The Lie algebra of $SU(2)^m$ is
$$
m\cdot \mathfrak{su}(2):=\mathfrak{su}(2)\oplus \cdots \oplus \mathfrak{su}(2)\hspace*{0.1in}(m~\mbox{times}).
$$
An element of $m\cdot \mathfrak{su}(2)$ is then an $m$-tuple $(X_1,\dots,X_m)$ with $X_i\in \mathfrak{su}(2)$ for $i=1,\dots, m$. The Lie algebra structure on $m\cdot \mathfrak{su}(2)$ is simply the direct sum Lie algebra structure:
$$
[X,Y]:=([X_1,Y_1],\dots, [X_m,Y_m]),
$$
where
$$
X=(X_1,\dots,X_m),\hspace*{0.1in}Y=(Y_1,\dots,Y_m).
$$
For $1\le j\le m$, let $\pi_j: m\cdot \mathfrak{su}(2)\rightarrow  \mathfrak{su}(2)$ be the natural projection from the $j$-th component of $m\cdot \mathfrak{su}(2)$ onto $\mathfrak{su}(2)$. Let $\mathfrak{su}(2)_j$ be the Lie subalgebra of $m\cdot \mathfrak{su}(2)$ defined by
$$
\mathfrak{su}(2)_j:=\{X\in m\cdot \mathfrak{su}(2)~|~\pi_k(X)=0,~\forall~k\neq j\}.
$$
Let $\phi_j:\mathfrak{su}(2)\hookrightarrow \mathfrak{su}(2)_j$ be the natural inclusion.  For $X\in \mathfrak{su}(2)$, we define
$$
X_{(j)}:=\phi_j(X)\in \mathfrak{su}(2)_j.
$$
On the other hand, for $X\in m\cdot \mathfrak{su}(2)$, we define
$$
X^{(j)}:=\phi_j\circ \pi_j(X)\in \mathfrak{su}(2)_j.
$$
Lastly, let $\widehat{\mathfrak{su}(2)}_j$ be the Lie subalgebra of $m\cdot \mathfrak{su}(2)$ defined by 
$$
\widehat{\mathfrak{su}(2)}_j:=\{X\in m\cdot \mathfrak{su}(2)~|~\pi_j(X)=0\}.
$$
Note that 
$$
m\cdot \mathfrak{su}(2)=\mathfrak{su}(2)_j\oplus \widehat{\mathfrak{su}(2)}_j,\hspace*{0.2in}[\mathfrak{su}(2)_j,\widehat{\mathfrak{su}(2)}_j]=0.
$$
For additional clarity, consider the case $m=4$ and let $X=(A,B,C,D)\in 4\cdot\mathfrak{su}(2)$.  For example, we have
$$
X^{(3)}=C_{(3)}=(0,0,C,0)\in \mathfrak{su}(2)_3
$$
and 
$$
(A,B,0,D)\in \widehat{\mathfrak{su}(2)}_3.
$$

\section{Proof of Theorem \ref{thmMainTheorem}}
\label{secProofMainThm}
\noindent In this section, we freely use the notation of Section \ref{secPrelim}.  The strategy we employ in the proof of Theorem \ref{thmMainTheorem} is to first derive a number of general properties of any left-invariant complex structure on $SU(2)^m$ using elementary algebraic arguments (where $m>1$ is necessarily an even number).  The majority of the proof is devoted to this task. While the algebra itself is easy to follow, there is a fair amount of work to be done here. Once we have derived these properties, we will use these properties to obtain a proof by contradiction for Theorem \ref{thmMainTheorem}.    

With that said, let $I$ be any left-invariant integrable almost complex structure on $SU(2)^m$. Using our standard $\mathfrak{su}(2)$-basis $\{e_1,e_2,e_3\}$, we obtain an $\mathfrak{su}(2)_j$-basis $\{e_{1(j)},~e_{2(j)},~e_{3(j)}\}$ for $j=1,\dots, m$.  Using the decomposition $m\cdot \mathfrak{su}(2)=\mathfrak{su}(2)_j\oplus \widehat{\mathfrak{su}(2)}_j$, we write
\begin{align}
    \label{eqAX}
    Ie_{1(j)}&=A_j+X_j,\\
    \label{eqBY}
    Ie_{2(j)}&=B_j+Y_j,\\
    \label{eqCZ}
    Ie_{3(j)}&=C_j+Z_j,
\end{align}
where $A_j,B_j,C_j\in \mathfrak{su}(2)_j$ and $X_j,Y_j,Z_j\in \widehat{\mathfrak{su}(2)}_j$.  In addition, decompose $A_j$, $B_j$, and $C_j$ as a linear combination of $e_{1(j)}$, $e_{2(j)}$, and $e_{3(j)}$:
\begin{align}
    \label{eqAj}
    A_j&=a_{1j}e_{1(j)}+a_{2j}e_{2(j)}+a_{3j}e_{3(j)},\\
    \label{eqBj}
    B_j&=b_{1j}e_{1(j)}+b_{2j}e_{2(j)}+b_{3j}e_{3(j)},\\
    \label{eqCj}
    C_j&=c_{1j}e_{1(j)}+c_{2j}e_{2(j)}+c_{3j}e_{3(j)}.
\end{align}
\begin{lemma}
    \label{lemSUj2dim}
    Let $\mathcal{E}_j:=\mbox{span}\{A_j,B_j,C_j\}$ and $\mathcal{F}_j:=\mbox{span}\{X_j,Y_j,Z_j\}$.  Then $\dim \mathcal{E}_j\ge 2$ and $\dim \mathcal{F}_j\ge 1$.
\end{lemma}
\begin{proof}
    Since $\dim \mathfrak{su}(2)_j=3$, it follows that $I\mathfrak{su}(2)_j\not\subset \mathfrak{su}(2)_j$, which implies that $\dim \mathcal{F}_j\ge 1$. 

    Suppose that $\dim \mathcal{E}_j\le 1$.  
    Then, without loss of generality, we may assume that $B_j=\lambda A_j$ and $C_j=\gamma A_j$ for some $\lambda,\gamma\in \mathbb{R}$.  Let 
    $$
    u:=e_{2(j)}-\lambda e_{1(j)},\hspace*{0.1in}v:=e_{3(j)}-\gamma e_{1(j)}.
    $$
    Then $0\neq [u,v]\in \mathfrak{su}(2)_j$,
    $$
    Iu=Y_j-\lambda X_j\in \widehat{\mathfrak{su}(2)}_j,\hspace*{0.1in}Iv=Z_j-\gamma X_j\in \widehat{\mathfrak{su}(2)}_j,
    $$
    and 
    \begin{align*}
        0&=N_I(u,v)\\
        &=I[Iu,v]+I[u,Iv]+[u,v]-[Iu,Iv]\\
        &=I0+I0+[u,v]-[Iu,Iv].
    \end{align*}
    Hence,
    $$
    [u,v]=[Iu,Iv]\in \mathfrak{su}(2)_j\cap \widehat{\mathfrak{su}(2)}_j=\{0\},
    $$
    which contradicts the fact that $[u,v]\neq 0$.  Hence, we must have $\dim \mathcal{E}_j\ge 2$.
\end{proof}
\begin{remark}
Examples of left-invariant complex structures can be \\constructed on $SU(2)^{m}$ for the case where $\mathcal{E}_j$ has dimension $2$ as well as the case where $\mathcal{E}_j$ has dimension $3$.  For the $\dim \mathcal{E}_j=2$ case, one has the left-invariant complex structure $J$ on $SU(2)\times SU(2)$ given by
    $$
    Je_{1(1)}=e_{2(1)},\hspace*{0.1in}Je_{2(1)}=-e_{1(1)},\hspace*{0.1in}Je_{3(1)}=e_{3(2)},
    $$
    $$
    Je_{1(2)}=e_{2(2)},\hspace*{0.1in}Je_{2(2)}=-e_{1(2)},\hspace*{0.1in}Je_{3(2)}=-e_{3(1)}.
    $$
    An example of a left-invariant complex structure on $SU(2)\times SU(2)$ for the $\dim \mathcal{E}_j=3$ case is given by
    $$
    J'e_{1(1)}=e_{2(1)},\hspace*{0.1in}J'e_{2(1)}=-e_{1(1)},\hspace*{0.1in}J'e_{3(1)}=e_{3(1)}+e_{3(2)},
    $$
    $$
    J'e_{1(2)}=e_{2(2)},\hspace*{0.1in}J'e_{2(2)}=-e_{1(2)},\hspace*{0.1in}J'e_{3(2)}=-2e_{3(1)}-e_{3(2)}.
    $$
    We will see later, however, that $\dim \mathcal{F}_j$ is always $1$.
\end{remark}
\begin{proposition}
    \label{propXYZdependence}
    The following statements are equivalent.
    \begin{itemize}
        \item[(i)] There exists $u\in \mathfrak{su}(2)_j$ such that $u\neq 0$ and $Iu\in \mathfrak{su}(2)_j$.
        \item[(ii)] There exists a unique 2-dimensional subspace $\mathcal{U}_j$ of $\mathfrak{su}(2)_j$ which is invariant under $I$.
        \item[(iii)] $\mathcal{F}_j:=\mbox{span}\{X_j, Y_j,Z_j\}$ has dimension $1$.
        \item[(iv)] $X_j$, $Y_j$, $Z_j$ are linearly dependent.
        \item[(v)] $[X_j,Y_j]=[Z_j,X_j]=0$.
        \item[(vi)] $[X_j,Y_j]=[Y_j,Z_j]=0$.
        \item[(vii)] $[Y_j,Z_j]=[Z_j,X_j]=0$.
    \end{itemize}
\end{proposition}
\begin{proof}
    $(i)\Rightarrow (ii)$. If $u$ is a nonzero element of $\mathfrak{su}(2)_j$ satisfying $Iu\in \mathfrak{su}(2)_j$, then one may define $\mathcal{U}_j$ to be the span of $u$ and $Iu$. For uniqueness, let $\mathcal{U}_j'$ be another $I$-invariant $2$-dimensional subspace of $\mathfrak{su}(2)_j$.  Since $\dim \mathfrak{su}(2)_j=3$, it follows that $\mathcal{U}_j\cap \mathcal{U}_j'\neq \{0\}$.  Let $u'\in \mathcal{U}_j\cap \mathcal{U}'_j$ be any nonzero element.  Then $\{u', Iu'\}\subset\mathcal{U}_j\cap \mathcal{U}_j'$ are linearly independent elements.  Since $\mathcal{U}_j$ and $\mathcal{U}_j'$ have dimension $2$, we must have $\mathcal{U}_j=\mathcal{U}'_j$. 
    
    $(i)\Leftarrow (ii)$. If $\mathcal{U}_j$ is a 2-dimensional $I$-invariant subspace of $\mathfrak{su}(2)_j$, then any nonzero element $u$ of $\mathcal{U}_j$ satisfies $Iu\in \mathcal{U}_j\subset \mathfrak{su}(2)_j$.

    $(ii)\Rightarrow (iii)$. Let $\{u_1,u_2\}$ be a basis of $\mathcal{U}_j$ and write
    \begin{align*}
    u_1&=\alpha_1e_{1(j)}+\alpha_2e_{2(j)}+\alpha_3e_{3(j)},\\
    u_2&=\beta_1e_{1(j)}+\beta_2e_{2(j)}+\beta_3e_{3(j)}.
    \end{align*}
    Then $Iu_1, Iu_2\in \mathcal{U}_j\subset \mathfrak{su}(2)_j$ implies
    \begin{align*}
        &\alpha_1 X_j +\alpha_2 Y_j +\alpha_3Z_j=0,\\
        &\beta_1 X_j +\beta_2 Y_j +\beta_3Z_j=0.
    \end{align*}
    Since $(\alpha_1,\alpha_2,\alpha_3)$ and $(\beta_1,\beta_2,\beta_3)$ are linearly independent elements of $\mathbb{R}^3$, it immediately follows that $\dim \mathcal{F}_j\le 1$.  However, by Lemma \ref{lemSUj2dim}, we also have $\dim \mathcal{F}_j\ge 1$.  Hence, we must have $\dim \mathcal{F}_j=1$.

    $(i)\Leftarrow (iii)$.  Without loss of generality, we may assume that $\{X_j\}$ is a basis of $\mathcal{F}_j$.  Then $Y_j=\lambda X_j$ for some $\lambda \in \mathbb{R}$ and $u:=e_{2(j)}-\lambda e_{1(j)}$ is a nonzero element of $\mathfrak{su}(2)_j$ such that $Iu\in \mathfrak{su}(2)_j$.

    $(iii)\Rightarrow (iv)$.  Clear.

    $(i)\Leftarrow (iv)$.  There exists $\alpha_1,\alpha_2,\alpha_3\in \mathbb{R}$ (not all zero) such that 
    $$
    \alpha_1 X_j+\alpha_2 Y_j+\alpha_3Z_j=0.
    $$
    Let $u:=\alpha_1 e_{1(j)}+\alpha_2 e_{2(j)}+\alpha_3e_{3(j)}$.  Then $Iu\in \mathfrak{su}(2)_j$.

    $(iii)\Rightarrow (v), (vi), (vii)$.  Immediate.

    $(i)\Leftarrow (v)$. Expanding $N_I(e_{1(j)},e_{2(j)})=0$ and using the assumption \\$[X_j,Y_j]=0$ gives
    $$
    I([Ie_{1(j)},e_{2(j)}]+[e_{1(j)},Ie_{2(j)}])=[A_j,B_j]-e_{3(j)}\in \mathfrak{su}(2)_j.
    $$
    Doing the same for $N_I(e_{3(j)},e_{1(j)})=0$ and using the assumption that $[Z_j,X_j]=0$ gives
    $$
    I([Ie_{3(j)},e_{1(j)}]+[e_{3(j)},Ie_{1(j)}])=[C_j,A_j]-e_{2(j)}\in \mathfrak{su}(2)_j.
    $$
    Note that from the Lie algebra structure on $m\cdot\mathfrak{su}(2)$ we have
    $$
    u:=[Ie_{1(j)},e_{2(j)}]+[e_{1(j)},Ie_{2(j)}]\in \mathfrak{su}(2)_j
    $$
    and 
    $$
    v:=[Ie_{3(j)},e_{1(j)}]+[e_{3(j)},Ie_{1(j)}]\in \mathfrak{su}(2)_j.
    $$
    Consequently, if $u\neq 0$ or $v\neq 0$, then we are done.  So suppose that $u=0$ and $v=0$.  Then
    \begin{align}
        \label{eqU}
        [Ie_{1(j)},e_{2(j)}]&=[Ie_{2(j)},e_{1(j)}]\\
        \label{eqV}
        [Ie_{3(j)},e_{1(j)}]&=[Ie_{1(j)},e_{3(j)}],
    \end{align}
    and
    \begin{align}
        \label{eqAB}
        [A_j,B_j]&=e_{3(j)},\\
        \label{eqCA}
        [C_j,A_j]&=e_{2(j)}.
    \end{align}
    Expanding (\ref{eqU}) and (\ref{eqV}) using (\ref{eqAX}), (\ref{eqBY}), (\ref{eqAj}), and (\ref{eqBj}) gives
    $$
    a_{1j}e_{3(j)}-a_{3j}e_{1(j)}=-b_{2j}e_{3(j)}+b_{3j}e_{2(j)}
    $$
    and 
    $$
    -c_{2j}e_{3(j)}+c_{3j}e_{2(j)}=-a_{1j}e_{2(j)}+a_{2j}e_{1(j)},
    $$
    which implies
    \begin{equation}
    \label{eqRelationsabc}
     a_{1j}=-b_{2j},~c_{3j}=-a_{1j},\hspace*{0.1in}a_{3j}=b_{3j}=c_{2j}=a_{2j}=0.
    \end{equation}
    Expanding (\ref{eqAB}) with the help of (\ref{eqRelationsabc}) gives
    $$
    a_{1j}b_{2j}e_{3(j)}=-a_{1j}^2e_{3(j)}=e_{3(j)},
    $$
    which implies $a_{1j}^2=-1$.  Since all of the coefficients in (\ref{eqAj}) are real, we have a contradiction.  Hence, we must have $u\neq 0$ or $v\neq 0$.

    $(i)\Leftarrow (vi)$ and $(i)\Leftarrow (vii)$. The proof is entirely similar to $(i)\Leftarrow (v)$.
\end{proof}
Our goal now is to show that $I$ satisfies at least one (and hence all) of the conditions listed in Proposition \ref{propXYZdependence}.  For this, we expand $N_I(e_{1(j)},e_{2(j)})=0$, $N_I(e_{2(j)},e_{3(j)})=0$, and $N_I(e_{3(j)},e_{1(j)})=0$ using (\ref{eqAX})-(\ref{eqCZ}) and (\ref{eqAj})-(\ref{eqCj}) and then collect the coefficients.  For $N_I(e_{1(j)},e_{2(j)})=0$, one obtains
\begin{align}
    \label{eqE1E2coeffE1}
    &a_{1j}c_{1j}-a_{3j}a_{1j}+b_{2j}c_{1j}-b_{3j}b_{1j}-a_{2j}b_{3j}+a_{3j}b_{2j}=0,\\
    \label{eqE1E2coeffE2}
    &a_{1j}c_{2j}-a_{3j}a_{2j}+b_{2j}c_{2j}-b_{3j}b_{2j}-a_{3j}b_{1j}+a_{1j}b_{3j}=0,\\
    \label{eqE1E2coeffE3}
    &a_{1j}c_{3j}-a_{3j}^2+b_{2j}c_{3j}-b_{3j}^2+1-a_{1j}b_{2j}+a_{2j}b_{1j}=0,\\
    \label{eqE1E2coeffOther}
    &-a_{3j}X_j-b_{3j}Y_j+(a_{1j}+b_{2j})Z_j-[X_j,Y_j]=0.
\end{align}
For $N_I(e_{2(j)},e_{3(j)})=0$, one obtains
\begin{align}
    \label{eqE2E3coeffE1}
    &-b_{1j}^2+b_{2j}a_{1j}-c_{1j}^2+c_{3j}a_{1j}+1-b_{2j}c_{3j}+b_{3j}c_{2j}=0,\\
    \label{eqE2E3coeffE2}
    &-b_{1j}b_{2j}+b_{2j}a_{2j}-c_{1j}c_{2j}+c_{3j}a_{2j}-b_{3j}c_{1j}+b_{1j}c_{3j}=0,\\
    \label{eqE2E3coeffE3}
    &-b_{1j}b_{3j}+b_{2j}a_{3j}-c_{1j}c_{3j}+c_{3j}a_{3j}-b_{1j}c_{2j}+b_{2j}c_{1j}=0,\\
    \label{eqE2E3coeffOther}
    &(b_{2j}+c_{3j})X_j-b_{1j}Y_j-c_{1j}Z_j-[Y_j,Z_j]=0.
\end{align}
For $N_I(e_{3(j)},e_{1(j)})=0$, one obtains
\begin{align}
    \label{eqE3E1coeffE1}
    &-c_{2j}c_{1j}+c_{3j}b_{1j}+a_{1j}b_{1j}-a_{2j}a_{1j}-c_{2j}a_{3j}+c_{3j}a_{2j}=0,\\
    \label{eqE3E1coeffE2}
    &-c^2_{2j}+c_{3j}b_{2j}+a_{1j}b_{2j}-a^2_{2j}+1-c_{3j}a_{1j}+c_{1j}a_{3j}=0,\\
    \label{eqE3E1coeffE3}
    &-c_{2j}c_{3j}+c_{3j}b_{3j}+a_{1j}b_{3j}-a_{2j}a_{3j}-c_{1j}a_{2j}+c_{2j}a_{1j}=0,\\
    \label{eqE3E1coeffOther}
    &-a_{2j}X_j+(a_{1j}+c_{3j})Y_j-c_{2j}Z_j-[Z_j,X_j]=0.
\end{align}
\begin{lemma}
    \label{lemXYZIndependence}
    If $X_j,Y_j,Z_j$ are linearly independent, then there exists $k\neq j$ such that $X^{(k)}_j,Y^{(k)}_j,Z^{(k)}_j$ is a basis of $\mathfrak{su}(2)_k$.  In addition, the following conditions hold 
    $$
    a_{1j}+b_{2j}\neq 0,\hspace*{0.1in}b_{2j}+c_{3j}\neq 0,\hspace*{0.1in}a_{1j}+c_{3j}\neq 0.
    $$
\end{lemma}
\begin{proof}
    Suppose $X_j,Y_j,Z_j$ are linearly independent.  By Proposition \ref{propXYZdependence}, at least two of the following brackets $[X_j,Y_j]$, $[Y_j,Z_j]$, and $[Z_j,X_j]$ are nonzero.  Without loss of generalitiy, assume that $[X_j,Y_j]\neq 0$.  Since $[X_j,Y_j]\in \widehat{\mathfrak{su}(2)}_j$, it follows that for some $k\neq j$, we have
    $$
    0\neq [X_j,Y_j]^{(k)}=[X^{(k)}_j,Y^{(k)}_j],
    $$
    where the last equality follows from the Lie algebra structure on $m\cdot\mathfrak{su}(2)$.  By Proposition \ref{propSU2XY}, $X^{(k)}_j$, $Y^{(k)}_j$, and $[X^{(k)}_j,Y^{(k)}_j]$ is basis of $\mathfrak{su}(2)_k$.  Taking the $k$-th component of (\ref{eqE1E2coeffOther}) gives
    $$
    [X^{(k)}_j,Y^{(k)}_j]=-a_{3j}X_j^{(k)}-b_{3j}Y^{(k)}_j+(a_{1j}+b_{2j})Z^{(k)}_j.
    $$
    It follows immediately from this that $X^{(k)}_j,Y^{(k)}_j,Z^{(k)}_j$ is also a basis of $\mathfrak{su}(2)_k$.  In particular, the above equation implies $a_{1j}+b_{2j}\neq 0$.

    In a similar way, we can apply the linear independence of $X^{(k)}_j,Y^{(k)}_j,Z^{(k)}_j$ to equations (\ref{eqE2E3coeffOther}) and (\ref{eqE3E1coeffOther}) to conclude that 
    $b_{2j}+c_{3j}\neq 0$ and $a_{1j}+c_{3j}\neq 0$.
\end{proof}
\begin{lemma}
    \label{lemABCequalities}
    If $X_j,Y_j,Z_j$ are linearly independent, then
    $$
    b_{1j}=a_{2j},\hspace*{0.1in}c_{2j}=b_{3j},\hspace*{0.1in}a_{3j}=c_{1j}.
    $$
\end{lemma}
\begin{proof}
    By Lemma \ref{lemXYZIndependence}, $X^{(k)}_j,Y^{(k)}_j,Z^{(k)}_j$ is a basis of $\mathfrak{su}(2)_k$ for some $k\neq j$.  Equations (\ref{eqE1E2coeffOther}), (\ref{eqE2E3coeffOther}), and (\ref{eqE3E1coeffOther}) imply
    \begin{align}
        \label{eqXYZJacobi1}
        [[X^{(k)}_j,Y^{(k)}_j],Z^{(k)}_j]&=-a_{3j}[X^{(k)}_j,Z^{(k)}_j]-b_{3j}[Y^{(k)}_j,Z^{(k)}_j],\\
        \label{eqXYZJacobi2}
        [[Y^{(k)}_j,Z^{(k)}_j],X^{(k)}_j]&=-b_{1j}[Y^{(k)}_j,X^{(k)}_j]-c_{1j}[Z^{(k)}_j,X^{(k)}_j],\\
        \label{eqXYZJacobi3}
        [[Z^{(k)}_j,X^{(k)}_j],Y^{(k)}_j]&=-a_{2j}[X^{(k)}_j,Y^{(k)}_j]-c_{2j}[Z^{(k)}_j,Y^{(k)}_j].
    \end{align}
    Equations (\ref{eqXYZJacobi1})-(\ref{eqXYZJacobi3}) together with the Jacobi identity imply
    \begin{align*}
        (b_{1j}-a_{2j})&[X^{(k)}_j,Y^{(k)}_j]+(c_{2j}-b_{3j})[Y^{(k)}_j,Z^{(k)}_j]\\
        &+(a_{3j}-c_{1j})[Z^{(k)}_j,X^{(k)}_j]=0.
    \end{align*}
    Corollary \ref{corBracketsLI} implies 
    $$
    b_{1j}=a_{2j},\hspace*{0.1in}c_{2j}=b_{3j},\hspace*{0.1in}a_{3j}=c_{1j}.
    $$
\end{proof}
\begin{lemma}
    \label{lemMoreIdentities}
    If $X_j,Y_j,Z_j$ are linearly independent, then
    \begin{align}
        \label{eqRankA1}
        &a_{3j}b_{2j}-a_{2j}b_{3j}=b_{2j}c_{1j}-b_{1j}c_{2j}=0,\\
        \label{eqRankB1}
        &a_{1j}c_{2j}-a_{2j}c_{1j}=a_{1j}b_{3j}-a_{3j}b_{1j}=0,\\
        \label{eqRankC1}
        &b_{1j}c_{3j}-b_{3j}c_{1j}=a_{2j}c_{3j}-a_{3j}c_{2j}=0,
    \end{align}
    and
    \begin{align}
        \label{eqRankD1}
        b_{2j}c_{3j}-b_{3j}c_{2j}=a_{1j}b_{2j}-a_{2j}b_{1j}=a_{1j}c_{3j}-a_{3j}c_{1j}=-1.
    \end{align}
\end{lemma}
\begin{proof}
    We now substitute the equalities of Lemma \ref{lemABCequalities} into equations (\ref{eqE1E2coeffE1})-(\ref{eqE1E2coeffOther}), (\ref{eqE2E3coeffE1})-(\ref{eqE2E3coeffOther}), and (\ref{eqE3E1coeffE1})-(\ref{eqE3E1coeffOther}).  For (\ref{eqE1E2coeffE1})-(\ref{eqE1E2coeffE2}), we obtain
    \begin{align}
        \label{eqRankA}
        &a_{3j}b_{2j}-a_{2j}b_{3j}=b_{2j}c_{1j}-b_{1j}c_{2j}=0,\\
        \label{eqRankB}
        &a_{1j}c_{2j}-a_{2j}c_{1j}=a_{1j}b_{3j}-a_{3j}b_{1j}=0.
    \end{align}
    For (\ref{eqE2E3coeffE2})-(\ref{eqE2E3coeffE3}), we obtain (\ref{eqRankA}) as well as
    \begin{align}
        \label{eqRankC}
        b_{1j}c_{3j}-b_{3j}c_{1j}=a_{2j}c_{3j}-a_{3j}c_{2j}=0.
    \end{align}
    For (\ref{eqE3E1coeffE1}) and (\ref{eqE3E1coeffE3}), we again obtain (\ref{eqRankC}) and (\ref{eqRankB}) respectively.

    Combining (\ref{eqE1E2coeffE3}) and (\ref{eqE2E3coeffE1}) gives
    \begin{equation}
        \label{eqRankD}
        b_{2j}c_{3j}-b_{3j}c_{2j}=a_{1j}b_{2j}-a_{2j}b_{1j}.
    \end{equation}
    Combining (\ref{eqE1E2coeffE3}) and (\ref{eqE3E1coeffE2}) gives
    \begin{equation}
        \label{eqRankE}
        a_{1j}c_{3j}-a_{3j}c_{1j}=a_{1j}b_{2j}-a_{2j}b_{1j}.
    \end{equation}
    Set 
    $$
    \lambda:=b_{2j}c_{3j}-b_{3j}c_{2j}=a_{1j}b_{2j}-a_{2j}b_{1j}=a_{1j}c_{3j}-a_{3j}c_{1j}.
    $$
    Equation (\ref{eqE1E2coeffE3}) then simplifies to $\lambda+1=0$, which completes the proof.
\end{proof}
\begin{proposition}
    \label{propNoRealSolutions}
    The system of equations given by Lemmas \ref{lemABCequalities} and \ref{lemMoreIdentities} have no real number solutions.  In particular, $X_j,Y_j,Z_j$ must always be linearly dependent.
\end{proposition}
\begin{proof}
    First, consider the case where $b_{1j}=a_{2j}=0$.  (Note that $b_{1j}=a_{2j}$ by Lemma \ref{lemABCequalities}.)  From Lemma \ref{lemMoreIdentities}, we have $a_{1j}b_{2j}=-1$.  In particular, $a_{1j}$ and $b_{2j}$ are nonzero.  This, together with Lemma \ref{lemABCequalities}, Lemma \ref{lemMoreIdentities}, and the assumption $b_{1j}=a_{2j}=0$, gives
    $$
    a_{3j}=c_{1j}=0,\hspace*{0.1in}c_{2j}=b_{3j}=0,
    $$
    and $b_{2j}c_{3j}=a_{1j}c_{3j}=-1$.  From this, we have
    $$
     b_{2j}=c_{3j}=\frac{-1}{a_{1j}}.\hspace*{0.1in}
    $$
    Hence,
    $$
    b_{2j}c_{3j}=\frac{1}{a_{1j}^2}>0,
    $$
    which contradicts $b_{2j}c_{3j}=-1$.  If we repeat the above calculation using the assumption that $c_{2j}=b_{3j}=0$ or $a_{3j}=c_{1j}=0$, we see again that there are no real number solutions to the combined system of equations given by Lemmas \ref{lemABCequalities} and \ref{lemMoreIdentities}.  (Again, note that by Lemma \ref{lemABCequalities}, we always have $c_{2j}=b_{3j}$ and $a_{3j}=c_{1j}$ under the assumption that $X_j,Y_j,Z_j$ are linearly independent.) 

    So now let us assume that $b_{1j}=a_{2j}$, $c_{2j}=b_{3j}$, and $a_{3j}=c_{1j}$ are all nonzero.  Then $a_{1j}b_{3j}-a_{3j}b_{1j}=0$ can be rewritten as
    $$
    a_{1j}=\frac{a_{3j}b_{1j}}{b_{3j}}.
    $$
    Likewise, $a_{3j}b_{2j}-a_{2j}b_{3j}=0$ can be rewritten as 
    $$
    b_{2j}=\frac{a_{2j}b_{3j}}{a_{3j}}.
    $$
    Then
    \begin{align*}
        a_{1j}b_{2j}-a_{2j}b_{1j}&=\frac{a_{3j}b_{1j}}{b_{3j}}\frac{a_{2j}b_{3j}}{a_{3j}}-a_{2j}b_{1j}\\
        &=b_{1j}a_{2j}-a_{2j}b_{1j}\\
        &=0,
    \end{align*}
    which contradicts $a_{1j}b_{2j}-a_{2j}b_{1j}=-1$.  This completes the proof.
\end{proof}

\begin{corollary}
    \label{corInvariant2Dspace}
    Let $I$ be any left-invariant integrable almost complex structure on $SU(2)^m$.  Then for all $j\in \{1,\dots,m\}$, there exists a unique 2-dimensional $I$-invariant subspace $\mathcal{U}_j\subset \mathfrak{su}(2)_j$.  
\end{corollary}
\begin{proof}
    The existence of an $I$-invariant 2-dimensional subspace $\mathcal{U}_j$ of $\mathfrak{su}(2)_j$ follows from Proposition \ref{propNoRealSolutions} and Proposition \ref{propXYZdependence}.  The uniqueness of $\mathcal{U}_j$ follows from Proposition \ref{propXYZdependence}.
\end{proof}

Let us now assume (for the time being) that $SU(2)^m$ admits left-invariant hypercomplex structures (where we necessarily require that $m=4n>1$). Let $(I,J,K)$ be a fixed left-invariant hypercomplex structure on $SU(2)^m$.  For $j\in \{1,\dots, m\}$, let $\mathcal{U}_j$, $\mathcal{V}_j$, and $\mathcal{W}_j$ be the unique $2$-dimensional invariant subspaces of $\mathfrak{su}(2)_j$ associated to $I$, $J$, and $K$ respectively.  For dimensional reasons, we have
$$
\dim \mathcal{U}_j\cap \mathcal{V}_j=\dim\mathcal{V}_j\cap \mathcal{W}_j=\dim \mathcal{W}_j\cap \mathcal{U}_j=1.
$$
For each $j$, fix a nonzero element $E_j\in \mathcal{U}_j\cap \mathcal{V}_j$.  Then
$$
\mathcal{U}_j=\mbox{span}\{E_j,~IE_j\},\hspace*{0.05in}\mathcal{V}_j=\mbox{span}\{E_j,~JE_j\}, \hspace*{0.05in}\mathcal{W}_j=\mbox{span}\{IE_j,~JE_j\}.
$$
Note that $E_j,~IE_j,~JE_j$ is a basis of $\mathfrak{su}(2)_j$ and $KE_j\not\in \mathfrak{su}(2)_j$.  
\begin{lemma}
    \label{lemHypercomplex1}
    Let $j,k\in \{1,\dots, m\}$ with $j\neq k$.  Then
    $$
    [KE_j,E_k]=\lambda_{jk}E_k,\hspace*{0.1in}[KE_j,IE_k]=\lambda_{jk}IE_k,\hspace*{0.1in}[KE_j,JE_k]=\lambda_{jk}JE_k
    $$
    for some $\lambda_{jk}\in \mathbb{R}$.  In other words, the restriction of the adjoint map $\mbox{ad}_{KE_j}$ to $\mathfrak{su}(2)_k$ is $\lambda_{jk}\mbox{id}_{\mathfrak{su}(2)_k}$.  In addition, 
    $$
    [KE_j,KE_k]=K\left(\lambda_{jk}E_k-\lambda_{kj}E_j\right).
    $$
\end{lemma}
\begin{proof}
    Expanding $N_I(JE_j,E_k)=0$ and $N_J(IE_j,E_k)$ respectively gives 
    $$
     I[KE_j,E_k]=[KE_j,IE_k],\hspace*{0.1in}J[KE_j,E_k]=[KE_j,JE_k],
    $$
    which implies 
    $$
    [KE_j,E_k]\in \mathcal{U}_k\cap \mathcal{V}_k=\mbox{span}\{E_k\}.
    $$
    Hence, $[KE_j,E_k]=\lambda_{jk}E_k$ for some $\lambda_{jk}\in \mathbb{R}$.  The first part of the lemma now follows immediately.  The last part of the lemma is obtained by expanding $N_K(E_j,E_k)=0$ and using the first part.
\end{proof}
\noindent We are now in a position to prove Theorem \ref{thmMainTheorem}.  
\begin{proof}
    Suppose $(I,J,K)$ is a left-invariant hypercomplex structure with notation as defined above.  Let $j\in \{1,\dots,m\}$.  Then since $KE_j\not\in \mathfrak{su}(2)_j$,  we have $(KE_j)^{(k)}\neq 0$ for some $k\neq j$. Since $E_k,~IE_k,~JE_k$ is a basis of $\mathfrak{su}(2)_k$, one of the brackets $[KE_j,E_k]$, $[KE_j,IE_k]$, or $[KE_j,JE_k]$ must be nonzero.  Actually, Lemma \ref{lemHypercomplex1} implies that all three brackets are nonzero.  By Lemma \ref{lemHypercomplex1}, the adjoint map $\mbox{ad}_{KE_j}$ restricted to $\mathfrak{su}(2)_k$ is simply $\lambda_{jk}\mbox{id}_{\mathfrak{su}(2)_k}$ where $\lambda_{jk}\neq 0$.  Hence,  
    $$
    [IE_k,[KE_j,E_k]]=\lambda_{jk}[IE_k,E_k],
    $$
    $$
    [KE_j,[E_k,IE_k]]=-\lambda_{jk}[IE_k,E_k],
    $$
    $$
    [E_k,[IE_k,KE_j]]=\lambda_{jk}[IE_k,E_k].
    $$
    By the Jacobi identity, the above three terms must sum to zero.  Instead, we obtain
    \begin{align*}
        [IE_k,&[KE_j,E_k]]+[KE_j,[E_k,IE_k]]+[E_k,[IE_k,KE_j]]\\
        &=\lambda_{jk}[IE_k,E_k]\neq 0,
    \end{align*}
    which is a contradiction.  Note that the last equality is nonzero since $\lambda_{jk}\neq 0$ and $IE_k$ and $E_k$ are linearly independent elements in $\mathfrak{su}(2)_k$ which implies that their bracket is nonzero.    
\end{proof}

\end{document}